\documentclass[11pt,leqno]{amsart}
\usepackage[margin=1in]{geometry}
\usepackage{amsmath, amsthm, amssymb}
\usepackage[colorlinks=true, pdfstartview=FitV, linkcolor=blue,citecolor=blue, urlcolor=blue]{hyperref}
\usepackage[T1]{fontenc}
\usepackage[utf8]{inputenc}

\usepackage{parskip}

\def\pr{\right )}
\def\le{\left (}
\def\d{\,\mathrm{d}}

\def\e{\varepsilon}
\def\R {\mathbb{R}}
\def\t{\tilde}
\newtheorem{proposition}{Proposition}%[section]
\newtheorem{theorem}[proposition]{Theorem}
\newtheorem{corollary}[proposition]{Corollary}
\newtheorem{lemma}[proposition]{Lemma}
\theoremstyle{definition}

\newtheorem{remark}[proposition]{Remark}

\numberwithin{equation}{section}
\title[Viscosity solutions with the optimal regularity]{Existence of viscosity solutions with the optimal regularity of a two-peakon Hamilton--Jacobi equation}

\author{Tomasz Cie\'{s}lak and Jakub Siemianowski}

\address{Institute of Mathematics, Polish Academy of Sciences, \'Sniadeckich 8, 00-656 Warszawa, Poland}
\email{cieslak@impan.pl, jsiem@mat.umk.pl}

\begin{document}

%%%%%%%%%%%%%%%%%%%%%%%%%%%%%%%%%%%%%%
\begin{abstract}
This work, which continues the research begun in \cite{CSS}, \cite{CW}, is devoted to the studies of a Hamilton--Jacobi equation with a quadratic and degenerate Hamiltonian, which comes from the dynamics of a multipeakon in the Camassa--Holm equation.
It is given by a quadratic form with a singular positive semi-definite matrix.
We increase the regularity of the value function considered in \cite{CSS}, which is known to be the viscosity solution.
We prove that for a two-peakon Hamiltonian such solutions are actually $1/2$-H\"{o}lder continuous in space and time-Lipschitz continuous. The time-Lipschitz regularity is proven in any dimension $N\geq 1$. Such a regularity is already known in the one-dimensional simplifications (see \cite{CW}), moreover it is the best possible, as shown in \cite{CSS, CW}.

\end{abstract}

\maketitle
\section{Introduction}

We consider the initial-value problem for the Hamilton--Jacobi equation of the form
\begin{equation}\label{eq.uniq:0.1}
\begin{cases}
u_t(x,t) + \frac{1}{2} E(x) \nabla u(x,t) \cdot\nabla u(x,t)  = 0, &x\in \R^N,\; t\in(0,T),\\
u(x,0) = g(x),&x\in \R^N,
\end{cases}
\end{equation}
where $E(x)$ is a symmetric $N\times N$ matrix with entries $E(x)_{i,j}=e^{-|x_i-x_j|}$, $i,j=1,\ldots,N$, $x=(x_1,\ldots,x_N)\in \R^N$.
An initial condition $g:\R\to \R$ is a bounded Lipschitz continuous function. Such a problem arises in the investigation of dynamics of a two-peakon special solutions to the Camassa--Holm equation, see for instance \cite{BSS, CH}.

The multipeakon is a solution of the form $u(x,t)=\sum_{i=1}^N p_i(t)e^{-|x-q_i(t)|}$, which satisfies Camassa--Holm equation provided $p_i, q_i$ are trajectories of the Hamiltonian system with the Hamiltonian given by the quadratic form $H(q,p):=1/2 E(q)p\cdot p$, $p=(p_1, p_2,\ldots, p_N), q=(q_1, q_2,\ldots, q_N)$. The evolution of a multipeakon was an area of many investigations during last years, besides already mentioned \cite{BSS, CH}, see for instance \cite{CGKM, K}, where geometric tools are used to trace the dynamics of multipeakon.
Further information can be found in the introductory part of \cite{CSS}.
The Hamilton--Jacobi equation is a natural description of the front propagation of multipeakon trajectories. This is an original source of our interest in \eqref{eq.uniq:0.1}. Our study continues our previous contribution in \cite{CSS, CW}.

Let us mention that \eqref{eq.uniq:0.1} seems additionally interesting since, due to its degeneracy, it is not covered by the classical theory of unique solvability of the Hamilton--Jacobi equation. Neither classical approach of viscosity solutions (see \cite{CEL, Evans}) nor viability theory methods (see \cite{FPR, galbraith}) seem directly applicable in this case. The matrix $E(x)$, $x=(x_1,x_2,\ldots,x_N)$, becomes singular once $x_i=x_j$ for $i\neq j$ but it is positive semi-definite, see \cite{CGKM} for a proof of this fact (when $x_i\neq x_j$ for $i\neq j$, matrix $E(x)$ is even positive definite).
Due to the symmetry and positive semi-definiteness of $E(x)$, the square root $\sqrt{E(x)}$ of the matrix $E(x)$ is well-defined. Thus, the multipeakon Hamiltonian may be rewritten as $E(x)\nabla u\cdot\nabla u=|\sqrt{E(x)} \nabla u|^2$. In such a formulation, the Hamilton--Jacobi initial-value problems are studied in \cite{CD}, see also the extensions in \cite{CDL}.
Again, approach of \cite{CD, CDL} requires Lipschitz regularity of $\sqrt{E(x)}$ in $x$. This is not the case, since $\sqrt{E(x)}$ is only $1/2$-H\"{o}lder continuous in $x$ as we shall see later on.

We show that the value function which is a visocsity solution to the Hamilton--Jacobi equation with the two-peakon Hamiltonian, constructed in \cite{CSS}, is $1/2$-H\"{o}lder continuous in space and time-Lipschitz continuous. This is the best possible regularity.
We know from \cite{CW} that even in the one-dimensional simplification higher regularity fails, see also \cite[Proposition 5.4]{CSS} for the explicit example. Last but not least, we show that time-Lipschitz continuity of a solution of \eqref{eq.uniq:0.1}, constructed in \cite{CSS}, holds in any dimension $N\geq 1$, see Theorem \ref{thm.reg:2} in Section \ref{regularity}. It is worth mentioning that the optimal regularity obtained in this paper is still not enough to prove the uniqueness of a viscosity solution. Actually, quite a standard use of a doubling variables method, yields global uniqueness provided solutions are $(1/2+\epsilon)$-H\"{o}lder continuous in the space variable. For more information on this issue as well as more sophisticated local-in-time uniqueness, we refer to \cite[Section 5]{CSS}.

Our main theorem reads.
\begin{theorem}\label{glowne}
Let $N=2$ and $g$ be bounded and Lipschitz continuous.
Then a bounded viscosity solution $u$ of \eqref{eq.uniq:0.1}, constructed in \cite{CSS}, admits the following regularity: there is $C>0$ such that for all $(y,t)$, $(\t y, \t t)\in\R^2\times [0,T]$
\[
|u(y,t) - u(\t y, \t t)|\leq C\le |y-\t y|^{1/2} + |t-\t t|\pr.
\]
Moreover, $u$ is differentiable a.e. in $\R^2\times [0,T]$ and satisfies \eqref{eq.uniq:0.1} at all points of differentiability.
\end{theorem}

\section{Preliminaries}\label{preliminaries}

In this section we recall necessary lemmas and key steps of our method developed in \cite{CSS}.

We begin with the reformulation of \eqref{eq.uniq:0.1}, which allows us to associate the value function with our Hamilton--Jacobi problem.
As mentioned in Introduction, since $E(x)$ is a symmetric positive semi-definite matrix, its unique symmetric positive semi-definite square root $\sqrt{E(x)}$ is well-defined. Next, we notice that
\begin{equation}\label{refor}
\frac{1}{2}E(x)\nabla u\cdot\nabla u=\frac{1}{2}|\sqrt{E(x)} \nabla u|^2=\max_{a\in\R^N}\left(a\cdot\sqrt{E(x)}\nabla u-\frac{1}{2}|a|^2\right),
\end{equation}
by the use of the Legendre transform.
It is known that $u(x,t)$ is a viscosity solution of the initial-value problem \eqref{eq.uniq:0.1} if and only if $w(x,t):= u(x,T-t)$ is a viscosity solution of the following terminal-value problem
\[
\begin{cases}
w_t(x,t) - \frac{1}{2} E(x) \nabla w(x,t) \cdot\nabla w(x,t)  = 0, &x\in \R^N,\; t\in(0,T),\\
w(x,T) = g(x),&x\in \R^N.
\end{cases}
\]
Using \eqref{refor}, we rewrite the above terminal-value problem as
\begin{equation}\label{eq.terminal}
\begin{cases}
w_t(x,t) +  \min_{a\in\R^N}\left(a\cdot\sqrt{E(x)}\nabla w(x,t)+\frac{1}{2}|a|^2\right)  = 0, &x\in \R^N,\; t\in(0,T),\\
w(x,T) = g(x),&x\in \R^N.
\end{cases}
\end{equation}
We associate with \eqref{eq.terminal} the value function $v:\R^N\times [0,T]\to \R$ as follows
\begin{equation}\label{eq:1.4}
v(y,t) := \inf\left \{\frac{1}{2}\int_{t}^T|\alpha(s)|^2\d s + g(x(T)) \right \}, \quad y\in \R^N,\, t\in [0,T],
\end{equation}
where the infimum is taken over all measurable controls $\alpha:[0,T]\to \R^N$ and all solutions of
\[
\tag{$P(\alpha,y,t)$}
\begin{cases}
x^\prime(s) = \sqrt{E(x(s))}\alpha(s),&\text{for } t < s < T,\\
x(t) = y.
\label{eq:1.3}
\end{cases}
\]
Since $\sqrt{E(\cdot)}$ is only $1/2$-H\"{o}lder continuous (see \eqref{eq.reg:20}), solutions $x(s)$ are not unique in general.
\begin{remark}\label{rem:2} (a) Note that the infimum in \eqref{eq:1.4} is finite, since $\alpha\equiv 0$ is an admissible control.
Furthermore, the infimum in \eqref{eq:1.4} does not change if we restrict ourselves to controls $\alpha\in L^2((0,T),\R^N)$.
For $\alpha\in L^2((t,T),\R^N)$ there is always an absolutely continuous solution $x:[t,T]\to\R^N$ of ($P(\alpha,y,t)$).
For more details concerning these issues we refer the reader to \cite{CSS}.

(b) For convenience we consider the terminal-value problem \eqref{eq.terminal} instead of the initial-value problem \eqref{eq.uniq:0.1} .
Note that the Lipschitz or the H\"{o}lder regularity of $w$ and $u$ are the same.

(c) We showed in \cite[Theorem 4.4]{CSS} that $v$ defined by \eqref{eq:1.4} is a continuous bounded viscosity solution of \eqref{eq.terminal} provided that $g:\R\to \R$ is bounded and Lipschitz continuous.
Moreover, it was also shown in \cite[Section 5]{CSS} that every continuous viscosity solution of \eqref{eq.uniq:0.1} is unique and Lipschitz continuous on some time interval $(0,\tau)$, where $\tau>0$ depends on the Lipschitz constants of $g$ and $E(\cdot)$.
In contrast, the results of Section \ref{regularity} concern the time-Lipschitz continuity and the $1/2$-H\"{o}lder continuity in space  (the latter in dimension $N=2$ only), but they are global.\hfill $\square$
\end{remark}

\section{Regularity}\label{regularity}
In this section we show that the value function $v$ defined by \eqref{eq:1.4} associated with the two-dimensional
terminal-value problem \eqref{eq.terminal} is $1/2$-H\"{o}lder continuous in space and time-Lipschitz continuous.
We emphasize that the obtained regularity is the best possible.
In fact, we obtain a time-Lipschitz continuity of a value function $v$ in any dimension $N\geq 1$.
As a consequence of the regularity estimates, the value function turns to be an almost everywhere differentiable viscosity solution to initial and terminal-value problems.

Let us say a few words concerning the technical aspects. Our proof is based on an optimal control reformulation \eqref{eq:1.4} and  estimates of the value function.
To avoid repeating the similar arguments in slightly different circumstances, an abstract Lemma \ref{lem.reg:5} is  proven.
This lemma is helpful in showing the full regularity proof in the  possibly shortest way. Moreover, since Lemma \ref{lem.reg:5} is formulated in any dimension $N\geq 1$, it can be used in a future work on higher dimensional regularity.

We assume that, for every $\xi\in \R^N$, $A(\xi)$ is $N\times N$ real matrix such that
\begin{equation}\label{eq.reg:56}
\begin{cases}
\text{the map}\quad\R^N\ni\xi\mapsto A(\xi)\in \R^{N\times N}\quad\text{is continuous and}\\
|A(\xi)|\leq B|\xi |,\qquad\xi \in \R^N,
\end{cases}
\end{equation}
for some positive constant $B$.
Moreover, we assume that $h:\R^N\to \R$ satisfies the Lipschitz condition
\begin{equation}\label{eq.reg:58}
|h(y) - h(\t y)|\leq L|y-\t y|,\qquad y,\,\t y\in \R^N.
\end{equation}
Let $u:\R^N\times [0,T]\to\R$ be given by
\begin{equation}\label{eq.reg:59}
u(y,t) := \inf\left \{ \int_t^T|\alpha(s)|^2\d s + h(x(T))\right\},
\end{equation}
where the infimum is taken over all square integrable controls $\alpha:[0,T]\to \R^N$ and all solutions of
\begin{equation}\label{aux.prob}
\begin{cases}
\dot{x}(s) = A(x(s))\alpha(s),&s\in (t,T),\\
x(t) = y.
\end{cases}(\footnote{Note that the right hand-side is a Carath\'{e}odory map with a $L^1$--majorant, so there always  exists at least one absolutely continuous solution of this Cauchy problem.})\tag{P$_{\alpha, y ,t}$}
\end{equation}

First, we prove the time-Lipschitz regularity of a value function $u$ defined by \eqref{eq.reg:59}. Notice that the solution of \eqref{eq.terminal}, given 
by \eqref{eq:1.4}, satisfies assumptions of the theorem below. Hence, the announced time-Lipschitz regularity of a solution of \eqref{eq.uniq:0.1} follows.
\begin{theorem}\label{thm.reg:2}
Let $A:\R^N \to \R^{N\times N}$ satisfy \eqref{eq.reg:56}, $h:\R^N\to \R$ satisfy \eqref{eq.reg:58} and $u$ be the value function given by \eqref{eq.reg:59}. There is $k>0$ such that for all $y\in \R^N$, $t$, $\t t\in [0,T]$
\[
|u(y,t) - u(y,\t t)| \leq k|t- \t t|.
\]
\end{theorem}
\begin{proof}
Fix $y\in \R^N$, $\e>0$ and $t$, $\t t\in [0,T]$.
Without loss of generality we assume that $t\leq \t t$.
We find a control $\alpha_0$ and the solution $x_0$ of (P$_{\alpha, y,\t t}$) such that
\[
u(y,\t t) \geq  \frac{1}{2}\int_{\t t}^T|\alpha_0(s)|^2\d s + h(x_0(T)) - \e.
\]
We define
\[
\begin{aligned}
x_1(s) :=
\begin{cases}
x_0(\t t), & s\in[t,\t t),\\
x_0(s), &s\in[\t t, T],
\end{cases}\;\text{ and }\;
\alpha_1(s) :=
\begin{cases}
0, &s\in [t,\t t),\\
\alpha_0(s), & s\in[\t t,T].
\end{cases}
\end{aligned}
\]
Then $x_1$ is a solution of (P$_{\alpha_1,y,t}$) so
\begin{equation}\label{eq.reg:42}
u(y,t) - u(y,\t t) \leq \frac{1}{2}\int_t^T|\alpha_1(s)|^2\d s + h(x_1(T)) - \frac{1}{2}\int_{\t t}^T|\alpha_0(s)|^2\d s- h(x_0(T))+ \e = \e .
\end{equation}

On the other hand, we find a control $\alpha_2$ and the solution $x_2$ of (P$_{\alpha_2,y,t}$) such that
\[
u(y,t) \geq \frac{1}{2}\int_{t}^T|\alpha_2(s)|^2\d s + h(x_2(T)) - \e.
\]
We define $x_3(s) := x_2(s-(\t t-t))$ and $\alpha_3(s):=\alpha_2(s-(\t t- t))$, for $s\in[\t t,T]$.
Then $x_3$ is a solution of (P$_{\alpha_3,y,\t t}$), so we have
\begin{equation}\label{eq.reg:43}
\begin{aligned}
u(y,\t t )- u(y,t) &\leq \frac{1}{2}\int_{\t t }^T|\alpha_3(s)|^2\d s+ h(x_3(T)) - \frac{1}{2}\int_t^T|\alpha_2(s)|^2\d s- h(x_2(T)) + \e \\
&\leq -\frac{1}{2}\int_{T-(\t t -t)}^T|\alpha_2(s)|^2\d s + L|x_2(T-(\t t-t)) - x_2(T)| + \e,
\end{aligned}
\end{equation}
where we used  \eqref{eq.reg:58}.
We estimate the middle term of the right hand-side using \eqref{eq.reg:56}
\begin{equation}\label{eq.reg:11}
\begin{aligned}
L|x_2(T-(\t t - t)) - x_2(T)|&= L\left | y + \int_t ^{T-(\t t-t)} A(x_2(s))\alpha_2(s)\d s - y - \int_t^T A(x_2(s))\alpha_2(s)\d s \right|\\
& \leq LB\int_{T-(\t t-t)}^T|\alpha_2(s)|\d s\leq LB(\t t -t)^{1/2}\le \int_{T-(\t t-t)}^T|\alpha_2(s)|^2\d s\pr^{1/2} \\
&\leq \frac{(LB)^2}{4\eta}(\t t - t) + \eta \int_{T-(\t t-t)}^T|\alpha_2(s)|^2\d s.
\end{aligned}
\end{equation}
We select $\eta = \frac{1}{2}$ and plug it into \eqref{eq.reg:43} to deduce
\begin{equation}\label{eq.reg:44}
u(y,\t t) - u(y,t)\leq \frac{(LB)^2}{2}(\t t -t) + \e.
\end{equation}
Combining \eqref{eq.reg:42} and \eqref{eq.reg:44} shows that
\[
|u(y,\t t) - u(y,t)|\leq \frac{(LB)^2}{2}|\t t -t| + \e.
\]
Since $\e >0$ is arbitrary, the assertion follows.
\end{proof}

The following technical lemma allows us to provide a shorter proof of regularity estimates. 
\begin{lemma}\label{lem.reg:5}
We assume that $A(\cdot)$, $h$ satisfy \eqref{eq.reg:56}, \eqref{eq.reg:58} (respectively) and $u$ is defined by \eqref{eq.reg:59}.
Let $y$, $\t y\in \R^N$, $t\in[0,T]$ and $a$, $b>0$, $\gamma\in (0,1]$ be fixed.
We define
\begin{equation}\label{eq.reg:50}
\t t = t + a|y-\t y|^\gamma.
\end{equation}
If there exists a measurable control $\t \alpha:[t,\t t]\to \R^N$ and a solution $\t x:[t,\t t]\to \R^N$ of the problem
\begin{equation}\label{eq.reg:46}
\begin{cases}
\t x^\prime(s) = A(\t x(s))\t \alpha(s),&s\in (t,\t t),\\
\t x(t) = \t y\quad\text{and}\quad \t x(\t t) = y,
\end{cases}
\end{equation}
such that
\begin{equation}\label{eq.reg:47}
\int_t^{\t t} |\t \alpha(s)|^2 \d s \leq b|y- \t y|^\gamma
\end{equation}
and
\begin{equation}\label{eq.reg:53}
\left |y - \t x(s)\right | \leq |y - \t y|,\qquad s\in (t,\t t),
\end{equation}
then
\[
|u(y,t) - u(\t y,t)|\leq c |y-\t y|^\gamma,
\]
where
\begin{equation}\label{eq.reg:60}
c := \frac{1}{2}b + L|y -\t y|^{1-\gamma} + \frac{1}{2}a(LB)^2
\end{equation}
does not depend on $t$.
\end{lemma}
\begin{remark}\label{rem.reg:1}
If $\gamma = 1$, then the constant $c$ given by \eqref{eq.reg:60} is global in a sense that it does not depend on the distance $|y -\t y|$.
Moreover, if we restrict ourselves to $|y-\t y|\leq \eta$, for some $\eta >0$, then $c$ in \eqref{eq.reg:60} is of the form
\begin{equation}\label{eq.reg:67}
c = \frac{1}{2}b + L\eta^{1-\gamma} + \frac{1}{2}a(LB)^2.
\end{equation}
\end{remark}
\begin{proof}
Fix $y$, $\t y\in \R^N$, $t\in [0,T]$ and $\e >0$.
We first estimate the term $u(\t y,t) - u(y,t)$.
There are a control $\alpha_0:[t,T]\to \R^N$ and a solution $x_0$ of (P$_{\alpha_0,y,t}$)  satisfying
\[
u(y,t) >  \frac{1}{2}\int_t^T |\alpha_0(s)|^2 \d s + h(x_0(T)) -\e.
\]
We consider two cases.
If $\t t \leq T$, then we define
\[
\begin{aligned}
x_1(s) :=
\begin{cases}
\t x(s), & s\in[t,\t t),\\
x_0(s -(\t t - t)), &s\in[\t t, T],
\end{cases}\qquad\text{and}\qquad
\alpha_1(s) :=
\begin{cases}
\t \alpha(s), & s\in[t,\t t),\\
\alpha_0(s -(\t t - t)), &s\in[\t t, T].
\end{cases}
\end{aligned}
\]
Since $x_0$ is a solution of (P$_{\alpha_0,y,t}$) and $\t \alpha$, $\t x$ satisfy \eqref{eq.reg:46} (in particular $\t x(t)=\t y$),
we deduce that $x_1$ is a solution of (P$_{\alpha_1, \t y, t}$).
Thus
\begin{equation}\label{eq.reg:48}
\begin{aligned}
u(\t y,t) - u(y,t) &\leq \frac{1}{2}\int_t^T|\alpha_1(s)|^2\d s + h(x_1(T)) -\frac{1}{2} \int_t^T|\alpha_0(s)|^2\d s - h(x_0(T))+ \e \\
&\leq \frac{1}{2}\int_t^{\t t}|\t \alpha(s)|^2 \d s + \frac{1}{2}\int_{\t t}^T|\alpha_0(s-(\t t-t))|^2\d s + h(x_0(T-(\t t-t)))\\
&\qquad -\frac{1}{2} \int_t^T|\alpha_0(s)|^2\d s - h(x_0(T))+\e\\
&\leq \frac{1}{2}b|y-\t y|^\gamma - \frac{1}{2}\int_{T-(\t t-t)}^T|\alpha_0(s)|^2 \d s + L\left |x_0(T) - x_0(T-(\t t- t))\right |+\e,
\end{aligned}
\end{equation}
where we incorporated \eqref{eq.reg:46} and \eqref{eq.reg:58}.
We deal with the last term using \eqref{eq.reg:56} as in \eqref{eq.reg:11} to obtain
\[
\begin{aligned}
L\left |x_0(T) - x_0(T-(\t t- t))\right|  \leq \frac{(LB)^2}{4\eta}(\t t -t ) + \eta \int_{T-(\t t-t)}^T|\alpha_0(s)|^2\d s.
\end{aligned}
\]
Select $\eta =\frac{1}{2}$ and plug the above inequality into \eqref{eq.reg:48} and use \eqref{eq.reg:50}
\begin{equation}\label{eq.reg:51}
u(\t y,t) - u(y,t) \leq \frac{1}{2}b|y-\t y|^\gamma + \frac{(LB)^2}{2}(\t t-t)+\e \leq \le \frac{1}{2}b +\frac{(LB)^2}{2}a\pr |y -\t y|^\gamma + \e.
\end{equation}

Let us now consider the other case $\t t  >T$. Since $\t x$ is a solution of (P$_{\t \alpha,\t y,t}$) we obtain
\begin{equation}\label{eq.reg:52}
\begin{aligned}
u(\t y,t ) -u(y,t) & \leq \frac{1}{2}\int_t^T|\t \alpha(s)|^2 \d s + h(\t x (T)) - \frac{1}{2}\int_t^T|\alpha_0(s)|^2\d s - h(x_0(T)) +\e\\
&\leq \frac{1}{2} b|y-\t y|^\gamma - \frac{1}{2}\int_t^T|\alpha_0(s)|^2\d s +L \left | \t x (T) - x_0(T)\right |+\e.
\end{aligned}
\end{equation}
We estimate the last term using \eqref{eq.reg:56} and \eqref{eq.reg:53}
\begin{equation}\label{eq.reg:62}
\begin{aligned}
L \left | \t x (T) - x_0(T)\right |&=L \left | \t x(T) - y - \int_t^TA(x_0(s))\alpha_0(s)\d s\right |\leq L\left |\t x(T) -y\right| + LB\int_t^T|\alpha_0(s)| \d s\\
&\leq L|y-\t y| + LB(T-t)^{1/2}\le \int_t^T|\alpha_0(s)|^2\d s\pr ^{1/2}\\
&\leq L|y-\t y| + \frac{(LB)^2}{4\eta}(T-t) + \eta \int _t^T|\alpha_0(s)|^2 \d s.
\end{aligned}
\end{equation}
Select $\eta =\frac{1}{2}$, to obtain
\begin{equation}\label{eq.reg:54}
\begin{aligned}
u(\t y,t ) -u(y,t)&\leq \frac{1}{2} b|y-\t y|^\gamma  + L |y-\t y| + \frac{(LB)^2}{2}(T-t)+\e\\
&\leq \le \frac{1}{2}b + L|y -\t y|^{1-\gamma} + \frac{(LB)^2}{2}a\pr |y-\t y|^\gamma+\e.
\end{aligned}
\end{equation}
If we combine estimates for both cases \eqref{eq.reg:51} and \eqref{eq.reg:54}, we obtain in general
\[
u(\t y,t ) -u(y,t)\leq\le \frac{1}{2}b + L|y -\t y|^{1-\gamma} + \frac{(LB)^2}{2}a\pr |y-\t y|^\gamma+\e.
\]
Since $\e >0$ is arbitrary, we finally deduce
\begin{equation}\label{eq.reg:55}
u(\t y,t ) -u(y,t)\leq\le \frac{1}{2}b + L|y -\t y|^{1-\gamma} + \frac{(LB)^2}{2}a\pr |y-\t y|^\gamma.
\end{equation}

To estimate the term $u(y,t) - u(\t y,t)$, we follow the above proof, but inverse the time in a sense.
Namely, instead of functions $\t \alpha$, $\t x $, we use $\overline \alpha$, $\overline x$ defined by
\[
\overline{\alpha}(s):= - \t\alpha(\t t-(s-t))\quad\text{and}\quad \overline{x}(s):= \t x(\t t- (s- t)),\qquad s\in [t,\t t].
\]
Then, after similar calculations as above, we obtain exactly the same constant as in \eqref{eq.reg:55}, i.e.,
\[
u(y,t) - u(\t y,t ) \leq \le \frac{1}{2}b + L|y -\t y|^{1-\gamma} + \frac{(LB)^2}{2}a\pr |y-\t y|^\gamma,
\]
which completes the proof.
\end{proof}

From now on, we assume that the dimension $N=2$ and we proceed with the space-regularity of the value function $v$ defined by \eqref{eq:1.4}.
Let us recall that $g:\R^2 \to \R$ is assumed to be bounded and Lipschitz continuous, i.e., there is $L>0$ such that
\begin{equation}\label{assumption.reg}
|g(y)-g(\t y)| \leq L|y - \t  y|,\qquad y,\t y \in \R^2.
\end{equation}

The proof is divided into a sequence of lemmas, but we need some preparation first.
Let us recall that, for $x\in \R^2$, the explicit form of a matrix $\sqrt{E(x)}$ is known,
\begin{equation}\label{eq.reg:20}
\sqrt{E(x)}:=
\frac{1}{2}
\le
\begin{matrix}
\sqrt{1 + e^{-|x_1-x_2|}} + \sqrt{1-e^{-|x_1 -x_2|}}  &  \sqrt{1 + e^{-|x_1-x_2|}} - \sqrt{1-e^{-|x_1 -x_2|}}\\
\sqrt{1 + e^{-|x_1-x_2|}} - \sqrt{1-e^{-|x_1 -x_2|}}  &  \sqrt{1 + e^{-|x_1-x_2|}} + \sqrt{1-e^{-|x_1 -x_2|}}
\end{matrix}
\pr,
\end{equation}
so there is $B>0$ such that
\begin{equation}\label{eq.reg:21}
|\sqrt{E(x)}\xi|\leq B |\xi|.
\end{equation}
Let us denote by $D$ the diagonal in $\R^2$,
\[
D := \{(y_1,y_2) \in \R^2 \mid y_1 = y_2\}.
\]
The inverse $\sqrt{E(x)}^{-1}$ is well-defined whenever $x=(x_1,x_2)\in\R^2\setminus D$ and in this case we have
\[
\sqrt{E(x)}^{-1} =
\frac{1}{2}\le
\begin{matrix}
\frac{1}{\sqrt{1 + \zeta}} + \frac{1}{\sqrt{1- \zeta}}  &  \frac{1}{\sqrt{1 +  \zeta}} - \frac{1}{\sqrt{1- \zeta}}\\
\frac{1}{\sqrt{1 +  \zeta}} - \frac{1}{\sqrt{1- \zeta}} & \frac{1}{\sqrt{1 +  \zeta}} + \frac{1}{\sqrt{1- \zeta}}
\end{matrix}
\pr,
\]
where $\zeta:= e^{-|x_1 - x_2|}$ for short.
For $x\in \R^2\setminus D$ and $\xi=(\xi_1,\xi_2) \in \R^2$ we calculate
\[
\sqrt{E(x)}^{-1}\xi = \frac{1}{2}
\le \begin{matrix}
\frac{1}{\sqrt{1 + \zeta}}\le \xi_1 + \xi_2\pr  + \frac{1}{\sqrt{1- \zeta}}\le \xi_1 - \xi_2\pr \\
\frac{1}{\sqrt{1 + \zeta}}\le \xi_1 + \xi_2\pr  - \frac{1}{\sqrt{1- \zeta}}\le \xi_1 - \xi_2\pr
\end{matrix} \pr,
\]
so using the explicit formula for $\zeta$
\begin{equation}\label{eq.reg:69}
\left|\sqrt{E(x)}^{-1}\xi \right|^2= \frac{1}{4}\le 2\frac{1}{1+e^{-|x_1-x_2|}}\le \xi_1 + \xi_2\pr^2 +2 \frac{1}{1-e^{-|x_1-x_2|}}\le \xi_1- \xi_2\pr ^2 \pr.
\end{equation}
We use the inequality
\[
\frac{1}{1-e^{-s}}\leq \frac{1+s}{s},\qquad s > 0,
\]
to obtain the upper bound
\begin{equation}\label{eq.reg:65}
\left|\sqrt{E(x)}^{-1}\xi \right|^2\leq  \frac{1}{2}\le \le \xi_1 + \xi_2\pr^2 + \frac{1+|x_1-x_2|}{|x_1-x_2|}\le \xi_1- \xi_2\pr ^2 \pr.
\end{equation}

We recall that by $v$ we denote a value function given by \eqref{eq:1.4}, which is a viscosity solution to \eqref{eq.terminal}. As explained in Section \ref{preliminaries}, a solution to \eqref{eq.uniq:0.1} is obtained from $v$ via an inverse of time.
\begin{lemma}\label{lem.reg:1}
There exists a constant $c_1>0$ such that for all $y\in D$, $\t y\in \R^2 \setminus D$ and $t\in [0,T]$
\[
|y-\t y|\leq 1 \implies  |v(y,t) - v(\t y,t)| \leq c_1|y - \t y|^{1/2}.
\]
\end{lemma}
\begin{proof}
Take $y\in D$ and $\t y\in \R^2 \setminus D$ with
\begin{equation}\label{zeta}
|y - \t y|\leq 1.
\end{equation}
Our strategy is to use Lemma \ref{lem.reg:5} with $N = 2$, $A(x) = \sqrt{E(x)}$, $h=g$, and then identify $u = v$. Let
\begin{equation}\label{eq.reg:16}
\tilde{t} := t + 2|y-\t y|^{1/2}.
\end{equation}
We are about to construct $\t \alpha:[t,\t t)\to \R^2$(\footnote{We intentionally do not care about the endpoints.}) and $\t x:[t,\t t]\to \R^2$ satisfying \eqref{eq.reg:46}--\eqref{eq.reg:53}.
Let us define
\begin{equation}\label{eq.reg:64}
\tilde{x}(s) = \sigma(s)\t y + (1-\sigma(s)) y,\qquad s\in [t,\t t],
\end{equation}
where
\begin{equation}\label{eq.reg:17}
\sigma(s) = \frac{\le |y-\t y|^{1/2} - \frac{1}{2}(s-t)\pr ^2}{|y - \t y|}\in [0,1],\qquad s\in [t,\t t].
\end{equation}
The above formulas imply \eqref{eq.reg:53} and, taking into account \eqref{eq.reg:16}, yield $\t x (t) = \t y$ and $\t x(\t t) = y$.
Moreover, one can check that $\t x$ satisfies
\begin{equation}\label{eq.reg:63}
\t x ^\prime(s) = \frac{1}{|y - \t y|^{1/2} -\frac{1}{2}(s-t)}(y-\t x(s)),\qquad s\in (t,\t t).
\end{equation}
We use the fact that $y\in D$ and write
\begin{equation}\label{rho}
\rho(s):= |\t x_1(s) - \t x_2(s)| = \sigma(s)|\t y_1 - \t y_2|
\end{equation}
for short, where $\t x (s) = (\t x_1(s), \t x_2(s))$.
Since $y\in D$ and by \eqref{eq.reg:17}, \eqref{zeta} we have
\begin{equation}\label{eq.reg:66}
\rho(s) \leq |\t y_1 - \t y_2|\leq |\t y_1 - y_1| + |y_2 - \t y_2| \leq \sqrt{2}|y-\t y|\leq \sqrt{2}.
\end{equation}
By the above and by the fact that $\t y \in \R^2\setminus D$, $\t x_1 (s) \neq \t x_2 (s)$, for $s\in [t,\t t)$,  so we may define
\[
\t \alpha (s) = \frac{1}{|y - \t y|^{1/2} -\frac{1}{2}(s-t)}\sqrt{E(\t x (s))}^{-1}(y-\t x(s)).
\]
By the definition of $\t \alpha$ and \eqref{eq.reg:63},  we have
\[
\sqrt{E(\t x(s))}\t \alpha(s) = \frac{1}{|y - \t y|^{1/2} -\frac{1}{2}(s-t)}(y-\t x(s)) = \t x^\prime (s),\qquad s\in[t,\t t),
\]
so \eqref{eq.reg:46} holds.

What is left, is to show that \eqref{eq.reg:47} is satisfied.
We use \eqref{eq.reg:64}, \eqref{eq.reg:65}, \eqref{eq.reg:17}, \eqref{rho}, \eqref{eq.reg:66} and the fact that $y\in D$, to deduce
\begin{equation}
\begin{aligned}
|\t\alpha(s)|^2 &= \frac{1}{\le |y - \t y|^{1/2} -\frac{1}{2}(s-t)\pr ^2}\left |\sqrt{E(\t x (s))}^{-1}(y-\t x(s))\right|^2 \\
&\leq \frac{\sigma^2(s)}{\le |y - \t y|^{1/2} -\frac{1}{2}(s-t)\pr ^2}\left |\sqrt{E(\t x (s))}^{-1}(y-\t y)\right|^2 \\
&\leq \frac{\le |y - \t y|^{1/2} -\frac{1}{2}(s-t)\pr ^2}{|y-\t y|^2}\frac{1}{2}\le (y_1-\t y_1 + y_2 - \t y_2)^2
+ \frac{1+\rho(s)}{\rho(s)}(\t y_1 - \t y_2)^2 \pr\\
&\leq \frac{1}{2}\frac{\le |y - \t y|^{1/2} -\frac{1}{2}(s-t)\pr ^2}{|y-\t y|^2}\le 2|y-\t y|^2 + \frac{1+\sqrt{2}}{\sigma(s)|\t y_1 - \t y_2|}(\t y_1 - \t y_2)^2\pr \\
&\leq \le |y - \t y|^{1/2} -\frac{1}{2}(s-t)\pr ^2 + \frac{(1+\sqrt{2})|\t y_1 - \t y_2|}{2|y - \t y|}\\
&\leq \le |y - \t y|^{1/2} -\frac{1}{2}(s-t)\pr ^2 + \frac{(1+\sqrt{2})\sqrt{2}|y - \t y|}{2|y - \t y|}\\
&= \le |y - \t y|^{1/2} -\frac{1}{2}(s-t)\pr ^2 + \frac{(1+\sqrt{2})}{\sqrt{2}}\leq |y-\t y| + \frac{(1+\sqrt{2})}{\sqrt{2}}.
\end{aligned}
\end{equation}
Here we used the fact that the function $f(s) := \le |y - \t y|^{1/2} -\frac{1}{2}(s-t)\pr ^2$ , $s\in [t,\t t]$ attains the maximum at $s = t$.
Consequently
\[
\int_t^{\t t}|\t \alpha(s)|^2 \d s\leq (\t t - t) \le  |y-\t y| + \frac{(1+\sqrt{2})}{\sqrt{2}}\pr \leq \le 4  + \sqrt{2} \pr|y-\t y|^{1/2},
\]
where we used \eqref{eq.reg:16} and \eqref{zeta}.
Hence all of the assumptions of Lemma \ref{lem.reg:5} are met.
This finishes the proof if we take into account \eqref{eq.reg:67}.
\end{proof}

\begin{lemma}\label{lem.reg:2}
There is $c_2>0$ such that for all $y$, $\t y\in D$ and $t\in [0,T]$
\[
|v(y,t) - v(\t y,t)| \leq c_2 |y - \t y|.
\]
\end{lemma}
\begin{proof}
Again, we intend to make use of Lemma \ref{lem.reg:5}.
We define
\begin{equation}\label{eq.reg:14}
\t t := t + |y -\t y| = t + \sqrt{2}|y_1 - \t y_1|,
\end{equation}
where the last equality follows by assumption that $y$, $\t y\in D$.
We are about to construct a measurable control $\t \alpha:[t,\t t]\to \R^2$ and a solution $\t x:[t,\t t]\to\R^2$ of the problem
\begin{equation}\label{eq.reg:28}
\begin{cases}
\t x^\prime(s) = \sqrt{E(\t x(s))}\t \alpha(s),&s\in[t,\t t],\\
\t x(t) = \t y\quad\text{and}\quad \t x(\t t) = y,
\end{cases}
\end{equation}
such that
\begin{equation}\label{eq.reg:29}
\int_t^{\t t} |\t \alpha(s)|^2 \d s \leq |y- \t y|.
\end{equation}

Let $\t x:[t,\t t]\to \R^2$ be the solution of
\begin{equation}\label{eq.reg:9}
\begin{cases}
\t x ^\prime(s) = \mathrm{sgn} (y_1 - \t y_1)\frac{1}{\sqrt{2}}
\le\begin{matrix}
1\\
1
\end{matrix}\pr,
&s\in[t,\t t],\\
\t x (t) =\t y.
\end{cases}
\end{equation}
Then
\[
\t x (s) = \t y  + (s-t) \mathrm{sgn}(y_1-\t y_1)\frac{1}{\sqrt{2}}\le\begin{matrix}
1\\
1
\end{matrix}\pr,\qquad s\in[t,\t t],
\]
so in particular $\t x(\t t) = y$.
Moreover, we have $\t x_1 (s) = \t x_2(s)$, for $s\in[t,\t t]$.
The last formula yields
\begin{equation}\label{eq.reg:12}
\sqrt{E(\t x(s))} = \frac{1}{\sqrt{2}}
\le \begin{matrix}
1 & 1 \\
1 & 1
\end{matrix}\pr,\qquad s \in [t,\t t],
\end{equation}
see \eqref{eq.reg:20}.
We define $\t \alpha:[t,\t t]\to \R^2$ by
\[
\t \alpha(s) = \mathrm{sgn}(y_1 - \t y_1)
\le\begin{matrix}
1\\
0
\end{matrix}\pr,
\qquad s\in [t,\t t].
\]
Then, since $|\t \alpha(s)| \equiv 1$ and $\t t - t =  |y - \t y|$, we obtain \eqref{eq.reg:29}.
By \eqref{eq.reg:12}, we have
\[
\sqrt{E(\t x(s))}\t \alpha(s)  = \mathrm{sgn}(y_1 - \t y_1)\frac{1}{\sqrt{2}}
\le \begin{matrix}
1\\
1
\end{matrix} \pr  = \t x^\prime(s), \qquad s \in[t,\t t],
\]
what completes the proof of \eqref{eq.reg:28}.
Now, conditions \eqref{eq.reg:50}--\eqref{eq.reg:53} are satisfied.
With Lemma \ref{lem.reg:5} in hand (see also Remark \ref{rem.reg:1}), we draw the desired conclusion.
\end{proof}

To conclude the general case in which $y, \t y$ are placed arbitrarily in $\R^2$, we consider three subsets
\begin{equation}\label{eq.reg:30}
\begin{aligned}
\Omega_1 &:= \left\{ (y_1,y_2)\in \R^2 \mid y_2 - y_1 \geq \frac{1}{2}\right\},\\
\Omega_2 &:= \left\{ (y_1,y_2)\in \R^2 \mid |y_2 - y_1| \leq \frac{1}{2}\right\},\\
\Omega_3 &:= \left\{ (y_1,y_2)\in \R^2 \mid y_2 - y_1 \leq -\frac{1}{2}\right\}.
\end{aligned}
\end{equation}
\begin{lemma}\label{lem.reg:3}
There is $c_3 >0$ such that for all $y$, $\t y \in \Omega_2$ and $t\in [0,T]$
\[
|y - \t y| \leq \frac{1}{2} \implies |v(y,t) - v(\t y, t)| \leq c_3|y-\t y|^{1/2}.
\]
\end{lemma}
\begin{proof}
Let us fix $y$, $\t y\in \Omega_2$ satisfying
\begin{equation}\label{eq.reg:68}
|y-\t y|\leq \frac{1}{2} .
\end{equation}
Without loss of generality we assume that neither $y$ nor $\t y$ is on diagonal $D$, since we have already proved Lemmas \ref{lem.reg:1} and \ref{lem.reg:2}.
We consider several cases.

\begin{center}
\textbf{Case I: both $y$ and $\t y$ are on the same side of the diagonal.}
\end{center}

\noindent The distance $d(y,D)$ from the point $y$ to the diagonal $D$ is given by $d(y,D) = \frac{1}{\sqrt{2}}|y_1-y_2|$.
\begin{center}
\textbf{Subcase I (a): $\frac{1}{\sqrt{2}}\min\{|y_1-y_2|,|\t y_1 - \t y_2|\}\leq |y - \t y|$.}
\end{center}

\noindent Due to the symmetry, without loss of generality we may assume that
\[
\min\{|y_1-y_2|,|\t y_1 - \t y_2|\} = |y_1-y_2|.
\]
Then there exists exactly one $\xi\in D$ (the projection of $y$  onto D) such that $|y-\xi| = d(y,D) =\frac{1}{\sqrt{2}}|y_1-y_2|$.
We use the above and the fact that $y\in \Omega_2$ to obtain
\[
|y - \xi| \leq \frac{1}{2\sqrt{2}} .
\]
Whereas
\[
|\t y- \xi| \leq \underbrace{|\t y - y|}_{\leq 1/2\text{ by assumption}} + \underbrace{|y - \xi|}_{\leq 1/2\sqrt{2}\text{ by the above}}\leq 1.
\]
In view of Lemma \ref{lem.reg:1}
\[
\begin{aligned}
|v(y,t) - v(\t y,t)| &\leq |v(y,t) - v(\xi,t)| + |v(\xi,t)-v(\t y,t)|\leq c_1|y-\xi|^{1/2} + c_1|\xi - \t y|^{1/2}.
\end{aligned}
\]
Recall we supposed that
\[
|y - \xi| = \frac{1}{\sqrt{2}}\min\{|y_1-y_2|,|\t y_1 - \t y_2|\}\leq  |y-\t y|,
\]
so
\[
|\xi - \t y| \leq |\xi - y| + |y-\t y|\leq 2|y-\t y|.
\]
Hence, we deduce
\begin{equation}\label{eq.reg:31}
|v(y,t)- v(\t y,t)| \leq \le c_1 + c_1\sqrt{2}\pr|y-\t y|^{1/2}.
\end{equation}

\begin{center}
\textbf{Subcase I (b): $\frac{1}{\sqrt{2}}\min\{|y_1-y_2|,|\t y_1 - \t y_2|\}> |y - \t y|$.}
\end{center}

\noindent In that case we proceed via Lemma \ref{lem.reg:5}.
We define
\begin{equation}\label{eq.reg:32}
\t t = t + |y -\t y|^{1/2}
\end{equation}
and construct a control $\t \alpha:[t,\t t]\to \R^2$ and a solution $\t x:[t,\t t]\to \R^2$ of the following problem
\begin{equation}\label{eq.reg:33}
\begin{cases}
\t x^\prime(s) = \sqrt{E(\t x(s))}\t \alpha(s),&s\in[t,\t t],\\
\t x(t) = \t y\quad\text{and}\quad \t x(\t t) = y,
\end{cases}
\end{equation}
such that
\begin{equation}\label{eq.reg:34}
\int_t^{\t t} |\t \alpha(s)|^2 \d s \leq \frac{3+\sqrt{2}}{2\sqrt{2}} |y- \t y|^{1/2}.
\end{equation}
Let $\t x$ be the solution of
\[
\begin{cases}
\t x^\prime(s) = \frac{1}{|y-\t y|^{1/2}}(y-\t y),\\
\t x(t) = \t y,
\end{cases}
\]
or, explicitly,
\[
\t x(s) = \t y + \frac{s-t}{|y-\t y|^{1/2}}(y-\t y),\quad s\in[t,\t t].
\]
Obviously, we have $\t x(\t t) = y$.
Since both $y$, $\t y$ are on the same side of the diagonal and
\begin{equation}\label{eq.reg:35}
\t x([t,\t t]) = \mathrm{conv}\{y,\t y\},
\end{equation}
$\sqrt{E(\t x(s))}^{-1}$ is well-defined for every $s\in [t,\t t]$ .
We define
\[
\t \alpha(s) := \frac{1}{|y-\t y|^{1/2}}\sqrt{E(\t x(s))}^{-1}(y-\t y), \quad s\in [t,\t t].
\]
Hence, \eqref{eq.reg:33} is satisfied.
We use \eqref{eq.reg:65} and obtain
\begin{equation}\label{eq.reg:37}
\begin{aligned}
|\t \alpha(s)|^2 &\leq \frac{1}{|y-\t y|}\left |\sqrt{E(\t x(s))}^{-1}(y-\t y)\right|^2 \\
&\leq \frac{1}{|y - \t y|}\frac{1}{2}\le (y_1 - \t y_1 + y_2 - \t y_2)^2 + \frac{1+|\t x_1(s)- \t x_2(s)|}{|\t x_1(s) - \t x_2 (s)|}(y_1 - \t y_1 - (y_2 - \t y_2))^2 \pr\\
&\leq \frac{1}{2}\frac{1}{|y-\t y|}\le 2|y-\t y|^2 + \frac{1+|\t x_1(s)- \t x_2(s)|}{|\t x_1(s) - \t x_2 (s)|}2|y-\t y|^2 \pr\\
&= |y - \t y| +\frac{1+|\t x_1(s)- \t x_2(s)|}{|\t x_1(s) - \t x_2 (s)|}|y - \t y|,
\end{aligned}
\end{equation}
where  we incorporated the inequality $|y_1 - \t y_1| + |y_2 - \t y_2|\leq \sqrt{2}|y - \t y|$.
By \eqref{eq.reg:35} and by the assumption that $y$, $\t y \in \Omega_2$, we have
\[
|\t x_1 (s) - \t x _2(s)| \leq \max\{ |y_1 - y_2|, |\t y_1 - \t y_2|\} \leq \frac{1}{2}.
\]
On the other hand, recalling \eqref{eq.reg:35} and the fact that we deal with the Subcase I (b)
\[
|\t x_1 (s) - \t x _2(s)| \geq \min\{ |y_1 - y_2|, |\t y_1 - \t y_2|\} > \sqrt{2}|y-\t y|.
\]
The above estimates, in view of \eqref{eq.reg:37}, yield
\[
|\t \alpha(s)|^2 \leq |y-\t y| + \frac{3}{2\sqrt{2}}
\]
what together with \eqref{eq.reg:68} and \eqref{eq.reg:32}  implies \eqref{eq.reg:34}.
By Lemma \ref{lem.reg:5} and \eqref{eq.reg:67}, there is some $c>0$ independent of $y$, $\t y $ and $t$ such that
\begin{equation}\label{eq.reg:38}
|v(y,t) - v(\t y,t)|\leq c|y-\t y|^{1/2}.
\end{equation}

Formulas \eqref{eq.reg:31} and \eqref{eq.reg:38} prove the assertion of Lemma \ref{lem.reg:3} in the case $y$ and $\t y$ are on the same side of the diagonal.

\begin{center}
\textbf{Case II: $y$ and $\t y$ are on both sides of the diagonal.}
\end{center}

In this case, there exists $\xi\in \R^2$ and  $\theta \in (0,1)$ such that
\[
\xi \in D \cap \mathrm{conv}\{y,\t y\}\quad\text{and}\quad\xi = (1-\theta) y  + \theta \t y.
\]
Then
\[
\max\{|y - \xi|, |\t  y-\xi|\}\leq |y- \t y|\leq \frac{1}{2},
\]
so we use Lemma \ref{lem.reg:1} and get
\[
\begin{aligned}
|v(y,t) - v(\t y,t)| &\leq |v(y,t) - v(\xi,t)| + |v(\xi,t) - v(\t y,t)|\leq c_1|y-\xi|^{1/2} + c_1|\xi - \t y|^{1/2}\\
&\leq 2c_1|y-\t y|^{1/2}.
\end{aligned}
\]
This completes our argument.
\end{proof}
Next, we deal with the case when $y$ and $\t y$ are both in $\Omega_1$ (or in $\Omega_3$).
\begin{lemma}\label{lem.reg:4}
There is $c_4>0$ such that for all $y$, $\t y\in \Omega_1$ (respectively $y$, $\t y\in\Omega_3$) and $t\in [0,T]$
\[
|v(y,t) - v(\t y, t)| \leq c_4|y-\t y|.
\]
\end{lemma}
\begin{proof}
Let $y$, $\t y\in\Omega_1$ (the proof in the case of $\Omega_3$ is analogous).
Once again we are  about to use Lemma \ref{lem.reg:5}.
We define
\begin{equation}\label{eq.reg:45}
\t t := t + |y-\t y|
\end{equation}
and construct $\t x :[t,\t t]\to \R^2$
\[
\t x(s) := \t y + \frac{s-t}{|y-\t y|}(y-\t y),\quad s\in[t,\t t].
\]
Obviously $\t x$ satisfies
\[
\begin{cases}
\t x^\prime(s) = \frac{1}{|y-\t y|}(y-\t y),\\
\t x(t) = \t y\quad\text{and}\quad\t x(\t t) = y.
\end{cases}
\]
Since $\t x ([t,\t t]) \subset \Omega_1$, we have
\begin{equation}\label{eq.reg:41}
|\t x_1(s) - \t x_2(s)| \geq \frac{1}{2}.
\end{equation}
Thus, we define
\[
\t \alpha(s) := \frac{1}{|y-\t y|}\sqrt{E(\t x(s))}^{-1}(y-\t y), \quad s\in [t,\t t],
\]
and see that
\[
\sqrt{E(\t x(s))}\t \alpha(s) = \frac{1}{|y-\t y|}(y-\t y) = \t x^\prime(s), \quad s\in [t,\t t].
\]
So far we have shown that assumptions \eqref{eq.reg:46} and \eqref{eq.reg:53} are met.
We only need to show that \eqref{eq.reg:47} holds.
We use \eqref{eq.reg:69} and \eqref{eq.reg:41} to deduce
\[
\begin{aligned}
|\t \alpha(s)|^2&= \frac{1}{|y-\t y|^2} \left | \sqrt{E(\t x(s))}^{-1}(y-\t y)\right |^2\\
&= \frac{1}{|y-\t y|^2}\frac{1}{2}\le \frac{(y_1 -\t y_1 + y_2  - \t y_2)^2}{1 + e^{-|\t x_1(s) - \t x_2(s)|}}  + \frac{(y_1 - \t y_1 - (y_2 - \t y_2))^2}{1-e^{-|\t x_1(s) - \t x_2(s)|}}\pr\\
&\leq \frac{1}{2}\frac{1}{|y-\t y|^2}\le 2|y-\t y|^2 +  \frac{2 | y- \t y|^2}{1- e^{-1/2}}\pr = 1 + \frac{1}{1-e^{-1/2}}.
\end{aligned}
\]
We recall \eqref{eq.reg:45} and obtain
\[
\int_t^{\t t} |\t \alpha(s)|^2 \d s \leq \le 1 + \frac{1}{1-e^{-1/2}}\pr |y - \t y|.
\]
Now Lemma \ref{lem.reg:5} (see also Remark \ref{rem.reg:1}) proves the assertion.
\end{proof}
Analyzing the proof of Lemma \ref{lem.reg:4}, we conclude the following claim.
\begin{corollary}\label{rozniczkowalnosc}
Fix $\delta>0$ and consider a region $\Omega(\delta)=\{y=(y_1,y_2)\in \R^2:|y_1-y_2|\geq \delta \}$.
Then there exists $C(\delta)>0$ such that, for all $y$, $\t y \in \Omega(\delta)$ and $t\in [0,T]$,
\[
|v(y,t)-v(\t y,t)|\leq C(\delta)|y-\t y|.
\]
\end{corollary}
We are in a position to conclude the H\"{o}lder continuity of a value function $v$ given by \eqref{eq:1.4}.
\begin{theorem}\label{thm.reg:1}
There is $c_5>0$ such that for all $y$, $\t y\in \R^2$ and $t\in [0,T]$
\[
|v(y,t)- v(\t y, t)|\leq c_5 |y- \t y|^{1/2}.
\]
\end{theorem}
\begin{proof}
Fix some $y$, $\t y\in \R^2$.

\begin{center}
\textbf{Case I: $|y-\t y|\geq \frac{1}{2}$.}
\end{center}
By \cite[Lemma 2.3]{CSS}, for all $y\in \R^2$, $t\in[0,T]$, we have $|v(y,t)|\leq \|g\|_{L^\infty}$.
Therefore, we get
\[
\begin{aligned}
|v(y,t) - v(\t y,t)| &\leq 2\|g\|_{L^\infty} = 2\|g\|_{L^\infty} \sqrt{2} \le \frac{1}{2}\pr ^{1/2}\leq 2 \sqrt{2}  \|g\|_{L^\infty} |y-\t y|^{1/2}.
\end{aligned}
\]

\begin{center}
\textbf{Case II: $|y-\t y|< \frac{1}{2}$. }
\end{center}

If $y$, $\t y\in \Omega_1\cup\Omega_3$ the result follows by Lemma \ref{lem.reg:4}.
If $y$, $\t y\in \Omega_2$, then the result follows by Lemma \ref{lem.reg:3}.
What is left is to consider the case $y\in \Omega_1$ and $\t y\in \Omega_2$ (the case $y\in \Omega_3$ and $\t y\in \Omega_2$ is symmetric).
In such a situation, we find $\xi\in \Omega_1\cap \Omega_2$ of the form
\[
\xi = (1-\theta)y + \theta \t y,
\]
for some $\theta\in (0,1)$.
We use the inequalities $|y-\xi|\leq |y-\t y|$, $|\xi - \t y| \leq |y-\t y| <\frac{1}{2}$
and Lemmas \ref{lem.reg:3} and \ref{lem.reg:4} to obtain
\[
\begin{aligned}
|v(y,t)-v(\t y,t)|&\leq |v(y,t) - v(\xi,t)| +|v(\xi,t) - v(\t y,t)| \leq c_4|y-\xi| + c_3 |\xi -\t y|^{1/2}\\
&\leq c_4|y-\t y| + c_3|y-\t y|^{1/2}\leq \left(\frac{1}{\sqrt{2}}c_4+c_3\right)|y-\t y|^{1/2}.
\end{aligned}
\]
Thus the proof of the theorem is completed.
\end{proof}

Finally, a full space-time regularity for the value function $v$ associated with \eqref{eq.terminal} in the dimension $N=2$ can be achieved.
The Corollary \ref{cor.reg:3} below is a consequence of Remark \ref{rem:2}, Theorems  \ref{thm.reg:2}, \ref{thm.reg:1} and Corollary \ref{rozniczkowalnosc}. On the other hand, since regularity results concerning \eqref{eq.terminal} translate into regularity of a solution to \eqref{eq.uniq:0.1}, one infers Theorem \ref{glowne} as a consequence of
Corollary \ref{cor.reg:3}.

\begin{corollary}\label{cor.reg:3}
Let $N=2$ and $g$ be bounded and Lipschitz continuous.
The value function $v$ defined by \eqref{eq:1.4} is a bounded viscosity solution of \eqref{eq.terminal} with the following regularity: there is $C>0$ such that for all $(y,t)$, $(\t y, \t t)\in\R^2\times [0,T]$
\[
|v(y,t) - v(\t y, \t t)|\leq C\le |y-\t y|^{1/2} + |t-\t t|\pr.
\]
Moreover, $v$ is differentiable a.e. in $\R^2\times [0,T]$ and $v$ satisfies \eqref{eq.terminal} at all points of differentiability.
\begin{proof}
The fact that $v$ is a bounded viscosity solution of \eqref{eq.terminal} is shown in Theorem 4.4 of \cite{CSS}.
The regularity issues are consequences of  Theorems \ref{thm.reg:2}, \ref{thm.reg:1} and Corollary \ref{rozniczkowalnosc}. It is standard that at points of the differentiability of $v$, the equation \eqref{eq.terminal} is satisfied, see, e.g., \cite[\S II, Proposition 1.9]{BCD}.
\end{proof}
\end{corollary}

\noindent
{\bf Acknowledgement.} Both authors were supported by the National Science Center of Poland grant SONATA BIS 7 number UMO-2017/26/E/ST1/00989.

\end{document}